\documentclass[12pt]{amsart}
\usepackage{amsthm}
\usepackage{amssymb}
\usepackage{amsmath}
\theoremstyle{plain}
\newtheorem{proposition}{Proposition}
\newtheorem{theorem}[proposition]{Theorem}

\newtheorem{corollary}[proposition]{Corollary}

\theoremstyle{definition}
\newtheorem{definition}[proposition]{Definition}

\theoremstyle{definition}

\newtheorem{remark}[proposition]{Remark}

\numberwithin{equation}{section}

\numberwithin{proposition}{section}

\gdef\myletter{}

\let\savetheequation\theequation
\def\theequation{\savetheequation\myletter}

\newcommand{\CC}{{\mathbb C}}

\newcommand{\RR}{{\mathbb R}}

\newcommand{\PP}{{\mathbb P}}

\renewcommand{\Im}{\mbox{Im}}

\renewcommand{\Re}{\mbox{Re}}

\renewcommand{\date}{\today}
\def \bar{\overline}

\begin{document}

\vskip 3mm

\title[Pluripotential energy and large deviation]{\bf Pluripotential energy and large deviation }  

\author{T. Bloom* and N. Levenberg}{\thanks{*Supported in part by an NSERC of Canada grant}}

\address{University of Toronto, Toronto, Ontario M5S 2E4 Canada}  
\email{bloom@math.toronto.edu}

\address{Indiana University, Bloomington, IN 47405 USA}

\email{nlevenbe@indiana.edu}

\begin{abstract}
We generalize results from \cite{BL} relating pluripotential energy with the electrostatic energy of a measure given in \cite{[BBGZ]}. As a consequence, we obtain a large deviation principle for a canonical sequence of probability measures on a nonpluripolar compact set $K\subset \CC^n$. This is a special case of a result of R. Berman \cite{B}. For $n=1$, we include a proof that uses only standard  techniques of weighted potential theory.

\end{abstract}

\maketitle

\section{\bf Introduction.} \label{sec:intro}  

In \cite{[BBGZ]}, Berman, Boucksom, Guedj and Zeriahi defined the notion of {\it electrostatic energy} $E^*(\mu)$ associated with a probability measure $\mu$ on a compact K\"ahler manifold $X$ of dimension $n$. In \cite{BL}, after specializing to the case of a compact set $K\subset \CC^n\subset \PP^n$ we studied two related functionals $J(\mu), \ W(\mu)$ and their weighted counterparts $J^Q(\mu), \ W^Q(\mu)$. The functionals involve weak-* approximations of $\mu$ by discrete measures. The relation is given as Corollaries 5.7 and 5.8 in \cite{BL}. They give another interpretation of $E^*(\mu)$ which we called {\it pluripotential energy}.

\smallskip
Our main goals in this paper are Theorems \ref{wobsolete}, \ref{obsolete} and \ref{ldp}: 
\smallskip
\begin{enumerate}
\item Theorems \ref{wobsolete} and \ref{obsolete} sharpen and clarify the relation between $E^*$ (Definition \ref{eleven}) and the $J,J^Q$ and $W,W^Q$ functionals (Definitions \ref{jwmu} and \ref{jwmuq}) for measures with compact support $K\subset \CC^n$. In \cite{BL} these functionals were defined using a compact, convex set $H\supset K$, although their values were independent of the choice of $H$. In this paper the functionals are defined directly in terms of $K$; moreover $K$ can be any nonpluripolar compact set.
\item The functionals $J,J^Q$ and $W,W^Q$ can be defined using either a ``$\limsup$'' or a ``$\liminf$'' and the equality of these two, which we demonstrate in Theorems \ref{wobsolete} and  \ref{obsolete}, leads immediately to a large deviation principle (LDP) in Theorem \ref{ldp} {\it with rate function given in terms of these functionals}. In this setting there is a second approach to a LDP using a Legendre transform which leads to the $E^*$ functional; hence the relation between $E^*$ and the $J,J^Q$ and $W,W^Q$ functionals may be deduced from the {\it uniqueness} of the rate function (see Remark \ref{equival2}). 
\item A key ingredient in the proofs is a deep result of Berman, Boucksom and Nystrom \cite{BBN} on the convergence of the empirical measures of Fekete points, restated here as Theorem \ref{frombbn}. The transition from an energy of measures to discrete approximations of this energy yielding a LDP -- without utilizing a Legendre transform -- can be achieved in other potential-theoretic situations when one has the analogue of Theorem 2.4, cf., \cite{blang} for Angelsco ensembles. The authors are preparing a subsequent work on a LDP in a more general ``vector energy''  setting in the univariate case.
\item The utilization of a {\it strong Bernstein-Markov measure} (Definition \ref{strong}) is crucial for either approach to the LDP. In Corollary \ref{allbm} we show that {\it any} compact set admits a Bernstein-Markov measure and thus (Corollary  \ref{strongbm}) any {\it nonpluripolar} compact set admits a strong Bernstein-Markov measure.
\end{enumerate}

Theorem \ref{ldp} states a LDP 
for a canonical sequence $\{\sigma_k\}$ of measures on $\mathcal
{M}(K)$, the space of probability measures on $K$. The measures $\sigma_k$, which form a determinantal process, are the push-forwards of measures on products of $K$ and they are canonical in the sense that in the univariate case, i.e., $K\subset \CC$, they include the joint probability distributions of the eigenvalues of ensembles of unitary matrices restricted to $K$. Robert Berman \cite{B} has given two proofs of such a LDP in a more general context. In this paper we also include two proofs. Our first proof is similar in spirit to Berman's first proof in the sense that one establishes an ``easy'' upper bound and then a less-trivial lower bound. Here the effort is concealed in Theorems \ref{wobsolete} and \ref{obsolete}. Precisely, Theorem \ref{wobsolete} shows the equivalence of the appropriate form of $E^*(\mu)$ and an $L^{\infty}-$type discretization $W(\mu)$ for $\mu \in \mathcal M(K)$; the equivalence with the $L^2-$type discretization $J(\mu)$ follows for a $J$ functional defined with a Bernstein-Markov measure $\nu$. From this latter equivalence the LDP is obtained if $\nu$ is a strong Bernstein-Markov measure using a standard result (Theorem 4.1.11 \cite{DZ}). As mentioned in item (3) above, this procedure can be exploited in other settings. Indeed, in section 6 we indicate how in the univariate setting our approach relies solely on standard techniques and results from weighted potential theory. 

The second proof is the same as in Berman's work and was described to us by S\'ebastien Boucksom. It relies on a result in \cite{DZ} which states under general conditions that there is a LDP if a certain functional is Gateau differentiable; moreover the rate function is given as a Legendre transform. {\it In this setting the rate function is given directly in terms of $E^*$}. The differentiability follows from a fundamental result of Berman and Boucksom \cite{BBnew} restated here as Theorem \ref{bbdiff}. Indeed, in the multivariate setting, Theorem \ref{frombbn} is ultimately a consequence of Theorem \ref{bbdiff}.

Ben Arous and Guionnet \cite{BeG}, building on work of Voiculescu, first established an LDP for the Gaussian unitary ensemble (GUE). Their method was extended to more general ensembles by Hiai and Petz \cite{HP} and other authors. In the univariate case, the method utilized in each of the two proofs of the LDP on compact sets presented here are different from that utilized in \cite{BeG}. Furthermore, given that there
exist strong Bernstein-Markov measures which are not absolutely continuous with respect to Lebesgue measure (see Remark \ref{bmrem} and Corollary \ref{strongbm}), these methods give a LDP in the univariate case
not included in \cite{BeG} or \cite{HP}. Berman \cite{B} has proved a LDP on closed and unbounded sets by modifying his second proof in the compact case. This includes the GUE setting (see Remark \ref{gueremark}). 

\smallskip

\noindent{\bf Acknowledgements}. We thank Robert Berman and S\'ebastien Boucksom for valuable correspondence. Special thanks are due to the referee for many useful comments and suggestions on organization and exposition which greatly improved the paper.

\section{\bf Preliminaries.} \label{sec:prelims}	

In this section, we summarize important results in \cite{BBnew}, \cite{[BBGZ]} and \cite{BBN} adapted to our $\CC^n$ setting. We write $L(\CC^n)$ for the set of all plurisubharmonic (psh) functions $u$ on $\CC^n$ with the property that $u(z) - \log |z|$ is bounded above as $|z| \to \infty$ and
$$ L^+(\CC^n)=\{u\in L(\CC^n): u(z)\geq \log^+|z| + C\}$$
where $C$ is a constant depending on $u$. For locally bounded psh functions, e.g., for $u\in L^+(\CC^n)$, the complex Monge-Amp\`ere operator $(dd^cu)^n$ is well-defined as a positive measure. Here we normalize our definition of $dd^c=\frac{i}{\pi}\partial \bar \partial$ so that 
$$\int_{\CC^n} (dd^cu)^n=1 \ \hbox{for all} \ u\in L^+(\CC^n).$$
For arbitrary $u\in L(\CC^n)$ one can define the weak-* limit
$$NP(dd^cu)^n:=\lim_{j\to \infty} \bigl( {\bf 1}_{\{u >-j\}}\cdot (dd^c\max[u, -j]) ^n\bigr)$$
(cf., \cite{bedtay}). For $u\in L(\CC^n)$ with \begin{equation} \label{nonpp} \int_{\CC^n} NP(dd^cu)^n=1,\end{equation}
we write $(dd^cu)^n:= NP(dd^cu)^n$. As in \cite{BL}, {\sl we only consider $u\in L(\CC^n)$ satisfying (\ref{nonpp})}. For such $u$, $(dd^cu)^n$ puts no mass on pluripolar sets. 

Let $K\subset \CC^n$ be compact -- an assumption remaining in force throughout the paper -- and let $Q$ be a lowersemicontinuous function with
$\{z\in K:e^{-Q(z)}>0\}$ nonpluripolar (note that this implies $K$ be nonpluripolar). We denote the collection of such $Q$ as $\mathcal A(K)$ and we call $Q$ an {\it  admissible weight}. We define the {\it weighted pluricomplex Green function} $V^*_{K
,Q}(z):=\limsup_{\zeta \to z}V_{K
,Q}(\zeta)$ where
$$V_{K
,Q}(z):=\sup \{u(z):u\in L(\CC^n), \ u\leq Q \ \hbox{on} \ K
\}. $$
The case $Q\equiv 0$ is the ``unweighted'' case and we simply write $V_K
$. We have $V^*_{K
,Q}\in L^+(\CC^n)$ and we call the measure 
$$\mu_{K,Q}:=(dd^cV_{K,Q}^*)^n,$$
which has support in $K$, the {\it weighted equilibrium measure}; if $Q\equiv 0$ we write $\mu_{K}=(dd^cV_{K}^*)^n$. An example is the $n-$torus $T=\{z\in \CC^n: |z_j|=1, \ j=1,...,n\}$ where
$$V_T(z) =\max_j [\log |z_j|, 0] \ \hbox{and} \ \mu_T=\frac{1}{2\pi}d\theta\times \cdots \times \frac{1}{2\pi}d\theta. $$
We say $K
$ is locally regular if for each $z\in K
$ the unweighted pluricomplex Green function for the set $K
\cap \overline{ B(z,r)}$ is continuous for $r=r(z)>0$ sufficiently small. Here $B(z,r)$ denotes the Euclidean ball with center 
$z$ and radius $r$. We say $K$ is $L-$regular if $V_K =V_K^*$; i.e., $V_K$ is continuous. If $K
$ is locally regular and $Q$ is continuous, then $V_{K,Q}$ is continuous. In general, it is known that 
\begin{equation}\label{supprop}\hbox{supp}(\mu_{K,Q})\subset \{z\in K: V_{K,Q}^*(z)\geq Q(z)\}\end{equation}
and that $V_{K,Q}^*=Q$ on $\hbox{supp}(\mu_{K,Q})$ q.e., i.e., except perhaps for a pluripolar set.

The strictly psh function $u_0(z):=\frac{1}{2}\log (1+|z|^2)$ belongs to the class $L^+(\CC^n)$. For $u\in L^+(\CC^n)$ define
\begin{equation}\label{c1form}E(u):=\frac{1}{n+1}\int_{\CC^n}\sum_{j=0}^n (u-u_0) (dd^c u)^j\wedge (dd^c u_0)^{n-j}.\end{equation} 
The functional $E$ is a primitive for the complex Monge-Amp\`ere operator (see Propositions 4.1 and 4.4 of \cite{BBnew}).

For $Q\in \mathcal A(K)$, define
 $$P(Q):=V_{K,Q}^*.$$
 The composition of the $E$ and $P$ operators is Gateaux differentiable; this non-obvious result (Theorem \ref{bbdiff}) was proved by Berman and Boucksom in \cite{BBnew}. It is the key ingredient in proving Theorem \ref{frombbn}.
 
 \begin{theorem}\label{bbdiff} The functional $E\circ P$ is Gateaux differentiable; i.e., for $K\subset \CC^n$ nonpluripolar and $Q\in \mathcal A(K)$, $F(t):= (E\circ P)(Q+tv)$ is differentiable for all $v\in C(K)$ and $t\in \RR$. Furthermore,
$$F'(0)= \int_K v (dd^cP(Q))^n.$$ 
 \end{theorem}

Next, for $k=1,2,...$ let $\mathcal P_k$ denote the space of holomorphic polynomials of degree at most $k$. We let dim${\mathcal P}_k=N_k={n+k\choose k}$ and $\sum_{j=1}^{N_k}deg(e_j)=\frac{n}{n+1}kN_k$ where $\{e_1,...,e_{N_k}\}$ is the standard monomial basis for ${\mathcal P}_k$. 
The weighted $k-$th order diameter of $K$ with $Q\in \mathcal A(K)$ is
$$\delta^{Q,k}(K):= \bigl(\max_{x_1,...,x_{N_k}\in K} |VDM_{k}(x_1,...,x_{N_k})| e^{-kQ(x_1)} \cdots e^{-kQ(x_{N_k})}\bigr)^{\frac{n+1}{nkN_k}}$$
where
\begin{equation}\label{vdmnot}VDM_{k}(x_1,...,x_{N_k})=\det [e_i(x_j) ]_{i,j=1,...,N_k}.\end{equation}
We introduce the shorthand notation
$$VDM_{k}^Q(x_1,...,x_{N_k}) :=VDM_{k}(x_1,...,x_{N_k}) e^{-kQ(x_1)} \cdots e^{-kQ(x_{N_k})}.$$ The limit  
\begin{equation}\label{wtdtd}\lim_{k\to \infty} \bigl(\delta^{Q,k}(K)\bigr)^{\frac{n}{n+1}}=:\bar \delta^Q (K)\end{equation}
exists (cf., \cite{BBnew} or \cite{BLtd}) and is called the (normalized) weighted transfinite diameter of $K$ with $Q\in \mathcal A(K)$; moreover $\bar \delta^Q(K)>0$ since $Q$ is admissible. The nonstandard $\frac{n}{n+1}$ in (\ref{wtdtd}) is to achieve a nicer {\it Rumely-type formula} (\cite{BBnew}, p. 383 and see \cite{rumely}):
\begin{equation}\label{keyrel}-\log \bar \delta^{Q}(K)= E( V_{K,Q}^*)-E(V_T).\end{equation}

\begin{definition} \label{eleven}The {\it electrostatic} or {\it pluripotential energy} $E^*(\mu)$ of a measure $\mu\in \mathcal M(K)$ from \cite{[BBGZ]} is a Legendre-type transform of the functional $E$:
  \begin{equation} \label{estar} E^*(\mu)=\sup_{Q\in  C(K)}[E( V_{K,Q}^*)-\int_K Qd\mu] +\int_Ku_0d\mu
  \end{equation}
(see also Proposition 5.6 in \cite{BL}). \end{definition}

In \cite{[BBGZ]} it is shown that $\mu_{K,Q}$ minimizes the functional
\begin{equation}\label{wtenmin} E^*(\mu)+\int_K (Q-u_0)d\mu \end{equation}
over all $\mu\in \mathcal M(K)$; the minimal value is $E(V^*_{K,Q})$. This gives a remarkable generalization of the univariate setting (see Remark \ref{remen} and section 6). Moreover, in \cite{[BBGZ]}, a variational approach to the following special case of a result of Guedj and Zeriahi \cite{GZ} was given.

\begin{theorem}\label{guedjzer} Let $\mu \in  \mathcal M(K)$ with $E^*(\mu)<+\infty$. Then there exists $u\in L(\CC^n)$ with $(dd^cu)^n=\mu$ and $\int_K ud\mu>-\infty$.

\end{theorem}	
	Finally, we recall a crucial result from \cite{BBN} on ``asymptotic weighted Fekete arrays."  

\begin{theorem} \label{frombbn} Let $K\subset {\CC}^n$ be nonpluripolar and let $Q\in \mathcal A(K)$. For each $k$, take $N_k$ points $x_1^{(k)},...,x_{N_k}^{(k)}\in K$ for which
$$\lim_{k\to \infty} \bigl(|VDM_k^Q(x_1^{(k)},...,x_{N_k}^{(k)})|\bigr)^{\frac{1}{kN_k}} =\bar \delta^Q(K).$$
Then $\nu_k:= \frac{1}{N_k}\sum_{j=1}^{N_k} \delta_{x_j^{(k)}}\to \mu_{K,Q}$ weak-*.
\end{theorem}

\section{\bf Bernstein-Markov Property.}\label{sec:bmp}

Given a compact set $K\subset \CC^n$ and a measure $\nu$ on $K$, we say that $(K,\nu)$ satisfies a Bernstein-Markov property if for all $p_k\in \mathcal P_k$, 	
$$||p_k||_K:=\sup_{z\in K} |p_k(z)|\leq  M_k||p_k||_{L^2(\nu)}  \ \hbox{with} \ \limsup_{k\to \infty} M_k^{1/k} =1.$$
More generally, for $K\subset \CC^n$ compact, $Q\in \mathcal A(K)$, and $\nu$ a measure on $K$, we say that the triple $(K,\nu,Q)$ satisfies a weighted Bernstein-Markov property if for all $p_k\in \mathcal P_k$, 
$$||e^{-kQ}p_k||_K \leq M_k ||e^{-kQ}p_k||_{L^2(\nu)} \ \hbox{with} \ \limsup_{k\to \infty} M_k^{1/k} =1.$$ 
If $K$ is locally regular and $Q$ is continuous, for $\nu = (dd^cV_{K,Q})^n$ it is known that $(K,\nu,Q)$ satisfies a weighted Bernstein-Markov property \cite{bloomweight}. For $Q=0$ and $K$ $L-$regular this was proved in \cite{NTZ}. 

For $\nu\in \mathcal M(K)$ and $Q\in \mathcal A(K)$, define 
\begin{equation}\label{wtdzn}Z_k:=Z_k(K,Q,\nu):= \int_K \cdots \int_K |VDM_k^Q(z_1,...,z_{N_k}
)|^2 d\nu(z_1) \cdots d\nu(z_{N_k}).\end{equation}
The following result is from \cite{BBnew} (see also \cite{BBN}).

\begin{proposition}
\label{weightedtd}
Let $K\subset \CC^n$ be a compact set and let $Q\in \mathcal A(K)$. If $\nu$ is a measure on $K$ with $(K,\nu,Q)$ satisfying a weighted Bernstein-Markov property, then 
\begin{equation}
\label{zeen}\lim_{k\to \infty} Z_k^{\frac{1}{2kN_k}
} =  \bar \delta^Q(K).\end{equation}\end{proposition}

Given $\nu$ as in Proposition \ref{weightedtd}, we define a probability measure $Prob_k$ on $K^{N_k}$ via, for a Borel set $A\subset K^{N_k}$,
\begin{equation}\label{probk}Prob_k(A):=\frac{1}{Z_k}\cdot \int_A  |VDM_k^Q(z_1,...,z_{N_k
})|^2 \cdot d\nu(z_1) \cdots d\nu(z_{N_k
}).\end{equation}
We immediately obtain the following (cf., Proposition 3.3 of \cite{BLtd}).
		
\begin{corollary} \label{largedev} Given $\eta >0$, define
 \begin{equation}\label{aketa}A_{k,\eta}:=\{(z_1,...,z_{N_k})\in K^{N_k}: |VDM_k^Q(z_1,...,z_{N_k})|^2 \geq (\bar \delta^Q(K) -\eta)^{2kN_k}\}.\end{equation}
Then there exists $k^*=k^*(\eta)$ such that for all $k>k^*$, 
$$Prob_k(K^{N_k}\setminus A_{k,\eta})\leq (1-\frac{\eta}{2\bar \delta^Q(K)})^{2kN_k}.$$
	\end{corollary}	
	
	\begin{proof} From Proposition \ref{weightedtd}, given $\epsilon >0$, 
	$$Z_k \geq [\bar \delta^Q(K) -\epsilon]^{2kN_k}$$
	for $k\geq k(\epsilon)$. Thus
	$$Prob_k(K^{N_k}\setminus A_{k,\eta})=$$
	$$\frac{1}{Z_k}\int_{K^{N_k}\setminus A_{k,\eta}} |VDM_k^Q(z_1,...,z_{N_k})|^2d\nu(z_1) \cdots d\nu(z_{N_k})$$
	$$\leq \frac{  [\bar \delta^Q(K) -\eta]^{2kN_k}} { [\bar \delta^Q(K) -\epsilon]^{2kN_k}}$$
	if $k\geq k(\epsilon)$. Choosing $\epsilon < \eta/2$ and $k^*=k(\epsilon)$ gives the result.
	\end{proof}
	
	Using (\ref{probk}), we get an induced probability measure ${\bf P}$ on the infinite product space of arrays $\chi:=\{X=\{x_j^{(k)}\}_{k=1,2,...; \ j=1,...,N_k}: x_j^{(k)} \in K\}$: 
	$$(\chi,{\bf P}):=\prod_{k=1}^{\infty}(K^{N_k},Prob_k).$$
	
\begin{corollary} \label{niceprob} 
Let $(K,\nu,Q)$ satisfy a weighted Bernstein-Markov property. For ${\bf P}$-a.e. array $X=\{x_j^{(k)}\}\in \chi$, 
$$\nu_k :=\frac{1}{N_k}\sum_{j=1}^{N_k}\delta_{x_j^{(k)}}\to \mu_{K,Q} \ \hbox{weak-*}.$$
\end{corollary}  
\begin{proof} (cf., \cite{hpbook}, p. 211) From Theorem \ref{frombbn} it suffices to verify for ${\bf P}$-a.e.
\begin{equation}\label{borcan}\liminf_{k\to \infty} \bigl(|VDM_k^Q(x_1^{(k)},...,x_{N_k}^{(k)})|\bigr)^{\frac{1}{kN_k}}  = \bar \delta^Q(K).\end{equation}
Given $\eta >0$, the condition that for a given array $X=\{x_j^{(k)}\}$ we have
$$\liminf_{k\to \infty} \bigl(|VDM_k^Q(x_1^{(k)},...,x_{N_k}^{(k)})|\bigr)^{\frac{1}{kN_k}} \leq \bar \delta^Q(K)-\eta$$
means that $(x_1^{(k)},...,x_{N_k}^{(k)})\in K^{N_k}\setminus A_{k,\eta}$ for infinitely many $k$. Thus setting 
$$E_k:=\{X \in \chi: (x_1^{(k)},...,x_{N_k}^{(k)})\in K^{N_k}\setminus A_{k,\eta}\},$$
we have
$${\bf P}(E_k)\leq Prob_k(K^{N_k}\setminus A_{k,\eta})\leq (1-\frac{\eta}{2\bar \delta^Q(K)})^{2kN_k}$$
and $\sum_{k=1}^{\infty} {\bf P}(E_k)<+\infty$. By the Borel-Cantelli lemma, 
$${\bf P}(\limsup_{k\to \infty} E_k)=0.$$ Thus, with probability one, only finitely many $E_k$ occur,  and (\ref{borcan}) follows.
\end{proof}

A stronger version of Corollary \ref{niceprob} will be given in section \ref{sec:ld}. We now show  that every compact set admits a measure satisfying a Bernstein-Markov property. Indeed, the following stronger statement is true.

\begin{proposition} \label{preallbm} Let $K\subset \RR^n$. There exists a measure $\nu \in \mathcal M(K)$ such that for all complex-valued polynomials $p$ of degree at most $k$ in the (real) coordinates $x=(x_1,...,x_n)$ we have
$$||p||_K\leq M_k ||p||_{L^2(d\nu)}$$
where $\lim_{k\to \infty}M_k^{1/k}=1$.

\end{proposition}

 \begin{proof}  To construct $\nu$, we first observe that if $K$ is a finite set, any measure $\nu$ which puts positive mass at each point of $K$ will work. If $K$ has infinitely many points, for each $k=1,2,...$ let $m_k =$dim$\mathcal P_k(K)$ where $\mathcal P_k(K)$ denotes the complex-valued polynomials on $\RR^{n}$ of degree at most $k$ restricted to $K$. Then $\lim_{k\to \infty} m_k =\infty$ and $m_k\leq {n+k\choose k}=0(k^n)$. For each $k$, let 
 $$\nu_k :=\frac{1}{m_k} \sum_{j=1}^{m_k} \delta (x_j^{(k)})$$
 where $\{x_j^{(k)}\}_{j=1,...,m_k}$ is a set of Fekete points of order $k$ for $K$; i.e., if $\{e_1,...,e_{m_k}\}$ is any basis for $\mathcal P_k(K)$, 
\begin{equation}\label{ratiovdm}\bigl |\det [e_i(x_j^{(k)})]_{i,j=1,...,m_k}\bigr|=\max_{q_1,...,q_{m_k}\in K}\bigl |\det[e_i(q_j)]_{i,j=1,...,m_k}\bigr|.\end{equation}
Define
 $$\nu:=c\sum_{k=1}^{\infty} \frac{1}{k^2}\nu_k$$
 where $c>0$ is chosen so that $\nu\in \mathcal M(K)$. 
 If $p\in \mathcal P_k(K)$, we have
 $$p(x)=\sum_{j=1}^{m_k} p(x_j^{(k)} )l_j^{(k)}(x)$$
 where, following the notation in (\ref{vdmnot}) with $N_k$ replaced by $m_k$ but using real variables,
 $$l_j^{(k)}(x)=\frac{VDM_k(x_1^{(k)},...,x_{j-1}^{(k)},x,x_{j+1}^{(k)},...,x_{m_k}^{(k)})}{\det [e_i(x_j^{(k)})]}\in \mathcal P_k(K)$$ so $l_j^{(k)}(x_j^{(k)})=\delta_{jk}$. Since $||l_j^{(k)}||_K=1$ from (\ref{ratiovdm}) we have 
 $$||p||_K \leq \sum_{j=1}^{m_k} |p(x_j^{(k)})|.$$
 On the other hand, 
 $$||p||_{L^2(d\nu)}\geq ||p||_{L^1(d\nu)}\geq  \frac{c}{k^2}\int_K |p|d\nu_k$$
 $$= \frac{c}{m_kk^2}\sum_{j=1}^{m_k} |p(x_j^{(k)})|.$$
 Thus we have
 $$||p||_K \leq {m_kk^2\over c} ||p||_{L^2(d\nu)}.$$
\end{proof}

Since the holomorphic polynomials in $\CC^n$ are a subset of the complex-valued polynomials in the underlying real coordinates in $\CC^n=\RR^{2n}$, we immediately obtain the following. 

\begin{corollary} \label{allbm} Let $K\subset \CC^n$ be a compact set. Then there exists a measure $\nu\in \mathcal M(K)$ such that $(K,\nu)$ satisfies a Bernstein-Markov property.
\end{corollary}

\begin{remark} \label{bmrem} Note that the measure $\nu$ constructed in Proposition \ref{preallbm} (and hence in Corollary \ref{allbm}) is discrete. In \cite{bl1} it was shown using a different procedure that any $L-$regular compact set admits a discrete measure satisfying a Bernstein-Markov property.
\end{remark}

\begin{definition} \label{strong} Given a compact set $K\subset \CC^n$, we say $\nu\in \mathcal M(K)$ satisfies a {\it strong Bernstein-Markov property} if $(K,\nu,Q)$ satisfies a weighted Bernstein-Markov property for all $Q\in  C(K)$.
\end{definition} 

 \begin{corollary} \label{strongbm} Let $K\subset \CC^n$ be a nonpluripolar compact set. Then there exists a measure $\nu\in \mathcal M(K)$ satisfying a strong Bernstein-Markov property.
 \end{corollary}
 
 \begin{proof} We consider $\CC^n=\RR^{2n}\subset \CC^{2n}$ and we use Corollary  \ref{allbm} to construct a measure $\nu$ on $K$ such that $(K,\nu)$ satisfies a Bernstein-Markov property with respect to holomorphic polynomials on $\CC^{2n}$. Theorem 3.2 of \cite{bloomweight} then shows that $(K,\nu,Q)$ satisfies a weighted Bernstein-Markov property for all $Q\in   C(K)$.
 \end{proof}

It is convenient to have a simple sufficient condition for a measure to satisfy the
strong Bernstein-Markov property. We say that $(K,\nu)$ satisfies a {\it mass-density property} if there exists $T>0$ so that
$$\nu(B(z_0,r))\geq r^T$$
for all $z_0\in K$ and all $r<r(z_0)$ where $r(z_0)>0$. In \cite{BLmass} it was shown that for $K$ $L-$regular, this property (indeed, a weaker mass-density property will suffice) implies that $(K,\nu)$ satisfies a Bernstein-Markov property. Thus if $K\subset \RR^n\subset \CC^n$ is $L-$regular and $(K,\nu)$ satisfies a mass-density property, then $(K,\nu)$ satisfies a strong Bernstein-Markov property. In particular, if $K=\bar D$ when $D$ is a bounded domain in $\RR^n$ with $C^1-$ boundary, any $\nu$ which is a positive, continuous multiple of Lebesgue measure on $D$ is a 
strong Bernstein-Markov measure for $K$.

\section{\bf Relation between $E^*$ and $J,J^Q, W,W^Q$ functionals.}
	\label{sec:eandp} 
We begin by sharpening Theorem 2.2 in \cite{BL} in two ways: 
\begin{enumerate}
\item the result applies to all  {\it nonpluripolar compact sets} $K\subset \CC^n$; 
\item the functionals defined using a ``$\limsup$'' and a ``$\liminf$'' coincide (see Definitions \ref{jwmu} and \ref{jwmuq}) -- and this is the essence of the large deviation principle in Theorem \ref{ldp}. 
\end{enumerate}

We remark that 
${\mathcal M}(K)$, with the weak-* topology, is a Polish space; i.e., a separable, complete metrizable space. A neighborhood basis 
of $\mu \in {\mathcal M}(K)$ is given by sets of the form 
\begin{equation}\label{nbhdbase}G(\mu, k, \epsilon) := \{\sigma  \in {\mathcal M}(K):
|\int_K (\Re z)^{\alpha}(\Im z)^{\beta} (d\mu - d\sigma )| < \epsilon\end{equation}
$$\hbox{for} \ 0\leq |\alpha|+|\beta| \leq k\} \ \hbox{where} \ \Re z =(\Re z_1,..., \Re z_n).$$

	Fix a nonpluripolar compact set $K$ and a strong Bernstein-Markov measure $\nu$. Given $G\subset {\mathcal M}(K)$, for each $s=1,2,...$ we set
\begin{equation}\label{nbhddef}\tilde G_s:= \{{\bf a} =(a_1,...,a_s)\in K^s: \frac{1}{s}\sum_{j=1}^s \delta_{a_j}\in G\}.\end{equation}
Then we define, for $k=1,2,...$, 
$$J_k(G):=[\int_{\tilde G_{N_k}}|VDM_k({\bf a})|^{2}d\nu ({\bf a})]^{1/2kN_k}$$
and
$$W_k(G):=\sup \{ |VDM_k({\bf a})|^{1/kN_k}: {\bf a} \in \tilde G_{N_k}\}.$$

\begin{definition} \label{jwmu} For $\mu \in \mathcal M(K)$ we define
$$\overline J(\mu):=\inf_{G \ni \mu} \overline J(G) \ \hbox{where} \ \overline J(G):=\limsup_{k\to \infty} J_k(G);$$
$$\underline J(\mu):=\inf_{G \ni \mu} \underline J(G) \ \hbox{where} \ \underline J(G):=\liminf_{k\to \infty} J_k(G);$$
and
$$\overline W(\mu):=\inf_{G \ni \mu} \overline W(G) \ \hbox{where} \ \overline W(G):=\limsup_{k\to \infty} W_k(G);$$
$$\underline W(\mu):=\inf_{G \ni \mu} \underline W(G) \ \hbox{where} \ \underline W(G):=\liminf_{k\to \infty} W_k(G).$$
\end{definition}

Here the infima are taken over all neighborhoods $G$ of the measure $\mu$ in ${\mathcal M}(K)$. 
Note that $\overline W,\underline W$ are independent of $\nu$ but, a priori, $\overline J,\underline J$ depend on $\nu$. These functionals are clearly nonnegative but can take the value zero. The weighted versions of these functionals are defined for $Q\in \mathcal A(K)$ starting with
\begin{equation}\label{jkqmu}J^Q_k(G):=[\int_{\tilde G_{N_k}}|VDM^Q_k({\bf a})|^{2}d\nu ({\bf a})]^{1/2kN_k}\end{equation}
and
$$W^Q_k(G):=\sup \{ |VDM^Q_k({\bf a})|^{1/kN_k}: {\bf a} \in \tilde G_{N_k}\}.$$

\begin{definition} \label{jwmuq} For $\mu \in \mathcal M(K)$ we define
$$\overline J^Q(\mu):=\inf_{G \ni \mu} \overline J^Q(G) \ \hbox{where} \ \overline J^Q(G):=\limsup_{k\to \infty} J^Q_k(G);$$
$$\underline J^Q(\mu):=\inf_{G \ni \mu} \underline J^Q(G) \ \hbox{where} \ \underline J^Q(G):=\liminf_{k\to \infty} J^Q_k(G);$$
and
$$\overline W^Q(\mu):=\inf_{G \ni \mu} \overline W^Q(G) \ \hbox{where} \ \overline W^Q(G):=\limsup_{k\to \infty} W^Q_k(G);$$
$$\underline W^Q(\mu):=\inf_{G \ni \mu} \underline W^Q(G) \ \hbox{where} \ \underline W^Q(G):=\liminf_{k\to \infty} W^Q_k(G).$$
\end{definition}

It is straightforward from the definitions (cf., \cite{BL}) that
\begin{enumerate}
\item $\overline J^Q(\mu)\leq \overline W^Q(\mu)\leq \bar \delta^Q (K)$ for $Q\in \mathcal A(K)$;
\item $\overline J(\mu)=\overline J^Q(\mu)\cdot e^{\int_K Qd\mu} \ \hbox{and} \  \overline W(\mu)=\overline W^Q(\mu)\cdot e^{\int_K Qd\mu}$ for $Q\in  C(K)$; hence
\item $\log \overline J(\mu)\leq \log \overline W(\mu)\leq \inf_{v\in  C(K)} [\log \bar \delta^{v} (K) +\int_K vd\mu]$; similarly
\item $\log \overline J^Q(\mu)\leq \log \overline W^Q(\mu)\leq \inf_{v\in  C(K)} [\log \bar \delta^{v} (K) +\int_K vd\mu]-\int_K Qd\mu$ for $Q\in   C(K)$.
\end{enumerate}

Properties (1)-(4) hold for the functionals $\underline J,\underline W,\underline J^Q,\underline W^Q$ using the same proofs as in section 3 of \cite{BL}. The uppersemicontinuity of all eight functionals on ${\mathcal M}(K)$ (with the weak-* topology) follows as well. We note the following.

 \begin{proposition} \label{cor44} The measure $\mu_{K,Q}$ is the unique maximizer of the functional $\mu \to \overline W^Q(\mu)$ and $\overline W^Q(\mu_{K,Q})=\bar \delta^Q(K)$.
 \end{proposition}
\begin{proof} It follows from its definition that 
$$\overline W^Q(\mu)=\sup\{\limsup_{k\to \infty} |VDM_{k}^{Q}({\bf a}^{(k)})|^{1/kN_k}\}\leq \bar \delta^Q(K)$$ 
where the supremum is taken over all arrays $\{{\bf a}^{(k)}\}_{k=1,2...}$ of $N_k-$tuples ${\bf a}^{(k)}$ whose normalized counting measures $\mu_k$ converges to $\mu$ weak*. If $\overline W^Q(\mu)=\bar \delta^Q(K)$ then there is an asymptotic weighted  
Fekete array $\{{\bf a}^{(k)}\}_{k=1,2...}$ as in the statement of Theorem \ref{frombbn} with $\mu_k\to  \mu$  weak*.  From this theorem, $\mu=\mu_{K,Q}$. 
\end{proof}

Theorems  \ref{wobsolete} and \ref{obsolete} show that the inequalities in (3) and (4) are equalities, and that the $\overline J,\overline W,\overline J^Q,\overline W^Q$ functionals coincide with their $\underline J,\underline W,\underline J^Q,\underline W^Q$ counterparts. 

 \begin{theorem} \label{wobsolete} Let $K\subset \CC^n$ be a nonpluripolar compact set and $Q\in  C(K)$. Then for any $\mu\in \mathcal M(K)$, 
 \begin{equation}\label{minwunwtd}\log \overline W(\mu)=\log \underline W(\mu)=\inf_{v\in  C(K)} [\log \bar \delta^{v} (K) +\int_K vd\mu] \ \hbox{and} \end{equation}
 \begin{equation}\label{minwwtd}\log \overline W^Q(\mu)=\log \underline W^Q(\mu)=\inf_{v\in  C(K)} [\log \bar \delta^{v} (K) +\int_K vd\mu]-\int_K Qd\mu.\end{equation}
  \end{theorem}
  
   \begin{remark} \label{remen} Note that using (\ref{estar}) and (\ref{keyrel}) equation (\ref{minwunwtd}) says that $$-\log \overline W(\mu)=-\log \underline W(\mu)=E^*(\mu)-\int_Ku_0d\mu-E(V_T)$$ while (\ref{minwwtd}) says that 
   $$-\log \overline W^Q(\mu)=-\log \underline W^Q(\mu)=E^*(\mu)+\int_K (Q-u_0)d\mu-E(V_T).$$ Thus from (\ref{wtenmin}) we recover Proposition \ref{cor44}.

\end{remark}  

\begin{proof} It suffices to prove (\ref{minwunwtd}) as then (\ref{minwwtd}) follows from property (2). We have the upper bound 
\begin{equation}\label{upperbound2}\log \underline W(\mu) \leq \log \overline W(\mu)\leq \inf_{v\in  C(K)} [\log \bar \delta^{v} (K) +\int_K vd\mu]\end{equation}
from (3).

For the lower bound, we consider cases.

\smallskip \noindent {\sl Case I: $\mu=\mu_{K,v}$ for some $v\in  C(K)$.}
\smallskip

 Using Theorem \ref{frombbn}, if we consider arrays of points $\{x_1^{(k)},...,x_{N_k}^{(k)}\}_{k=1,2,...}$ in $K$ for which
$$\lim_{k\to \infty} \bigl(|VDM_k^{v}(x_1^{(k)},...,x_{N_k}^{(k)})|\bigr)^{\frac{1}{kN_k}} =\bar \delta^{v}(K),$$
we have $\frac{1}{N_k}\sum_{j=1}^{N_k} \delta_{x_j^{(k)}}\to \mu_{K,v}$ weak-*. Thus for any neighborhood $G$ of $\mu_{K,v}$ we have $\bar \delta^{v}(K)\leq  \underline W^{v}(\mu_{K,v},G)$; hence 
\begin{equation}\label{newstep}  \underline W^{v}(\mu_{K,v})=\overline W^{v}(\mu_{K,v}) = \bar \delta^{v}(K). \end{equation}
Applying (2) to (\ref{newstep}) we obtain (\ref{minwunwtd}) for $\mu=\mu_{K,v}$:
\begin{equation}\label{wnewstep}\log \underline W(\mu_{K,v})=\log \overline W(\mu_{K,v}) = \log \bar \delta^{v}(K) +\int_K v d\mu_{K,v}.\end{equation}

\smallskip \noindent {\sl Case II: $\mu\in \mathcal M(K)$ with the property that $E^*(\mu)<\infty$.}
\smallskip

From Theorem \ref{guedjzer} there exists $u\in L(\CC^n)$ with $\mu =(dd^cu)^n$ and $\int_K ud\mu>-\infty$ (see also Proposition 5.6 of \cite{BL}). However, since $u$ is only usc on $K$, $\mu$ is not necessarily of the form $\mu_{K,v}$ for some $v\in  C(K)$. Following the argument in Proposition 4.3 of \cite{BL}, taking a sequence of functions $\{Q_j\}\subset  C(K)$ with $Q_j \downarrow u$ on $K$, the weighted extremal functions $V^*_{K,Q_j}$ decrease to $u$ on $\CC^n$; 
$$\mu_j:=(dd^cV_{K,Q_j})^n\to \mu=(dd^cu)^n \ \hbox{weak-}*;$$
and
\begin{equation} \label{needit} \lim_{j\to \infty} \int_K V^*_{K,Q_j}d\mu_j =\lim_{j\to \infty} \int_K V^*_{K,Q_j}d\mu=\int_K u d\mu.\end{equation}
(Note that in this case, since $V_{K,Q_j}$ is not necessarily continuous, we get $\tilde u:=\lim_{j\to \infty} V^*_{K,Q_j}\geq u$ on $\CC^n$ and $\tilde u=u$ q.e. on $K$ which suffices for the application of the domination principle, Corollary A.2 of \cite{BL}). From the previous case (\ref{wnewstep}), we have 
$$\log \overline W(\mu_j)= \log  \underline W(\mu_j)=\log \bar \delta^{Q_j} (K) +\int_K Q_jd\mu_j.$$
Using uppersemicontinuity of the functional $\mu \to \underline W(\mu)$,
$$\limsup_{j\to \infty} \underline W(\mu_j)=\limsup_{j\to \infty} \overline W(\mu_j)\leq \underline  W(\mu).$$
Since $Q_j \downarrow u$ on $K$, 
\begin{equation}\label{upbound}\limsup_{j\to \infty} \log \bar \delta^{Q_j}(K)=\lim_{j\to \infty}  \log \bar \delta^{Q_j}(K)\end{equation}
exists. Since $V^*_{K,Q_j}=Q_j$ q.e. on supp$\mu_j$, from (\ref{needit}), 
$$\lim_{j\to \infty} \int_K Q_jd\mu_j= \int_K u d\mu.$$
Thus, together with (\ref{upbound}), we see that
$$\lim_{j\to \infty} \log \underline W(\mu_j)=\lim_{j\to \infty}\bigl( \log \bar \delta^{Q_j} (K) +\int_K Q_jd\mu_j \bigr):=M$$
exists and is less than or equal to $\log \underline W(\mu)$. 
We want to show that 
\begin{equation}\label{finbound}\inf_{v} [\log \bar \delta^{v} (K) +\int_K vd\mu]\leq M.\end{equation}

By monotone convergence, $\lim_{j\to \infty} \int_K Q_j d\mu = \int_K ud\mu$. 
Since we also have $\lim_{j\to \infty}\int_K Q_j d\mu_j = \int_K ud\mu$, given $\epsilon >0$, for $j\geq j_0(\epsilon)$,
$$\int_K Q_j d\mu_j \geq \int_K Q_j d\mu -\epsilon \ \hbox{and} \ \log \underline W(\mu_j)< M+\epsilon.$$
Hence for such $j$,
$$\inf_{\rm w} [\log \bar \delta^{v} (K) +\int_K vd\mu]\leq \log \bar \delta^{Q_j} (K) +\int_K Q_jd\mu$$
$$\leq \log \bar \delta^{Q_j} (K) +\int_K Q_jd\mu_j+\epsilon=\log \underline W(\mu_j)+\epsilon < M+2\epsilon,$$
yielding (\ref{finbound}). This finishes the proof in Case II. Note that from (\ref{estar}) we have
$$\inf_{v} [\log \bar \delta^{v} (K) +\int_K vd\mu]>-\infty$$
in this case.

\smallskip \noindent {\sl Case III: $\mu\in \mathcal M(K)$ with the property that $\inf_{v} [\log \bar \delta^{v} (K) +\int_K vd\mu]=-\infty$.}
\smallskip

As in \cite{BL}, we must show that $\overline W(\mu)=0$. This follows trivially from the upper bound (\ref{upperbound2}).
 \end{proof}

We next turn to the $\overline J,\underline J, \overline J^Q$ and $\underline J^Q$ functionals. The key step in the proof of Theorem \ref{obsolete} is to verify (\ref{newstep}) for $\overline J^v(\mu_{K,v})$ and $\underline J^v(\mu_{K,v})$.

 \begin{theorem} \label{obsolete} Let $K\subset \CC^n$ be a nonpluripolar compact set and let $\nu$ satisfy a strong Bernstein-Markov property. Fix $Q\in  C(K)$. Then for any $\mu\in \mathcal M(K)$, 
 \begin{align}\label{minunwtd}\log \overline J(\mu)=\log \overline W(\mu)= \log \underline  J(\mu)=\log \underline W(\mu)\end{align}
 $$=\inf_{v\in  C(K)} [\log \bar \delta^{v} (K) +\int_K v d\mu]$$
and
 \begin{equation}\label{minwtd}\log \overline J^Q(\mu)=\log \overline W^Q(\mu)= \log \underline J^Q(\mu)=\log \underline W^Q(\mu)\end{equation}
 $$=\inf_{v\in  C(K)} [\log \bar \delta^{v} (K) +\int_K v d\mu]-\int_K Qd\mu.$$   \end{theorem}
 \begin{proof} As in the previous proof, it suffices to prove (\ref{minunwtd}) since (\ref{minwtd}) follows from property (2). Again we have the upper bound 
 $$\log \overline J(\mu)\leq \log \overline W(\mu)\leq \inf_{v\in  C(K)} [\log \bar \delta^{v} (K) +\int_K vd\mu]$$
 from (3); for the lower bound, we need only consider Case I where $\mu=\mu_{K,v}$ for $v\in  C(K)$. We show the analogue of (\ref{wnewstep}) for $\overline J,\underline J$: 
\begin{equation}\label{jversion}\log \overline J(\mu_{K,v})=\log \underline J(\mu_{K,v})=\log \bar \delta^{v} (K) +\int_K vd\mu_{K,v}.\end{equation}
This shows, in particular, that 
$$\log \overline J(\mu_{K,v})=\log \overline W(\mu_{K,v})= \log \underline  J(\mu_{K,v}u)=\log \underline W(\mu_{K,v})$$
and hence we have
$$\log \overline J(\mu)=\log \overline W(\mu)= \log \underline  J(\mu)=\log \underline W(\mu)$$
for arbitrary $\mu \in  \mathcal M(K)$ following the proof of Theorem \ref{wobsolete}. This proves (\ref{minunwtd}). To prove (\ref{jversion}), we first verify the following.
 \medskip
 
\noindent {\sl Claim: Fix a neighborhood $G$ of $\mu_{K,v}$. We can find $k_0$ such that   \begin{equation}\label{setinclu}  A_{k,1/k} \subset \tilde G_{N_{k}} \ \hbox{for all } \ k\geq k_0
 \end{equation}
 where $A_{k,1/k}$ is defined in (\ref{aketa}) with $Q=v$ and $\eta =1/k$}.
 \medskip
  
 \noindent We prove (\ref{setinclu}) by contradiction: if false, then there is a sequence $\{k_j\}$ with $k_j\uparrow \infty$ such that for all $j$ sufficiently large we can find a point $x^j=(x_1^j,...,x_{N_{k_j}}^j)\in 
 A_{k_j,1/k_j} \setminus \tilde G_{N_{k_j}}$. But $\mu_j:=\frac{1}{N_{k_j}}\sum_{i=1}^{N_{k_j}}\delta_{x_i^j}\not\in \tilde G_{N_{k_j}}$ for $j$ sufficiently large contradicts Theorem \ref{frombbn} since $\mu_j\in A_{k_j,1/k_j} $ and $k_j \uparrow \infty$ imply $\mu_j\to \mu_{K,v}$ weak-*. This proves the claim.
 \medskip
 
 We observe that since $N_k={n+k\choose k}$
 $$(1-\frac{1/k}{2\bar \delta^{{v}}(K)})^{2k N_{k}}\to 0 \ \hbox{as} \ k\to \infty.$$
Using (\ref{setinclu}) and Corollary \ref{largedev},
$$\frac{1}{Z_{k}} \int_{\tilde G_{N_{k}}}  |VDM_{k}^{v}(z_1,...,z_{N_{k}
})|^2 \cdot d\nu(z_1) \cdots d\nu(z_{N_{k}
})$$
$$\geq \frac{1}{Z_{k}} \int_{A_{k,1/k} }  |VDM_{k}^{v}(z_1,...,z_{N_{k}
})|^2 \cdot d\nu(z_1) \cdots d\nu(z_{N_{k}
})$$
$$\geq 1- (1-\frac{1/k}{2\bar \delta^{{v}}(K)})^{2k N_{k}}\to 1 \ \hbox{as} \ k\to \infty.$$
Since $\nu$ satisfies a strong Bernstein-Markov property and $v\in C(K)$, using Proposition \ref{weightedtd} we conclude that
$$\liminf_{k\to \infty} \frac{1}{2k N_{k}}\log \int_{\tilde G_{N_{k}}}  |VDM_{k}^{v}(z_1,...,z_{N_{k}
})|^2 d\nu(z_1) \cdots d\nu(z_{N_{k}
})$$
$$\geq \log \bar \delta^{{v}}(K).$$

Taking the infimum over all neighborhoods $G$ of $\mu_{K,v}$ we obtain
$$\log \underline J^{v}(\mu_{K,v})\geq \log \bar \delta^{{v}}(K).$$
Thus we have the version of (\ref{newstep}) with $\overline J^{v}$ and $\underline J^{v}$:
\begin{equation}\label{jeqn}\log \underline J^{v}(\mu_{K,v})=\log \overline J^{v}(\mu_{K,v})= \log \bar \delta^{{v}}(K).\end{equation}
Using (2) with $\mu = \mu_{K,v}$ we obtain (\ref{jversion}).

 \end{proof}
 
\begin{remark} \label{fourfour} From now on, we simply use the notation $J,J^Q,W,W^Q$ without the overline or underline. \end{remark}

\section{\bf Large deviation.}\label{sec:ld} In this section we take $K\subset \CC^n$ a nonpluripolar compact set and a measure $\nu$ satisfying a strong Bernstein-Markov property. Fix $Q\in  C(K)$. For $x_1,...,x_{N_k}\in K$, we get a discrete probability measure $\kappa_k(x):=\frac{1}{N_k}\sum_{j=1}^{N_k} \delta_{x_j}$. Define 
$j_k:  K^{N_k} \to \mathcal M(K)$ via 
$$j_k(x_1,...,x_{N_k})=\kappa_k(x):=\frac{1}{N_k}\sum_{j=1}^{N_k} \delta_{x_j}.$$
From (\ref{probk}), 
$\sigma_k:=(j_k)_*(Prob_k) $ is a probability measure on $\mathcal M(K)$: for a Borel set $B\subset \mathcal M(K)$,
\begin{equation}\label{sigmak}\sigma_k(B)=\frac{1}{Z_k} \int_{\tilde B_{N_k}} |VDM_k^Q(x_1,...,x_{N_k
})|^2 d\nu(x_1) \cdots d\nu(x_{N_k
})\end{equation}
(recall (\ref{probk}) and (\ref{nbhddef}); here, $Z_k:=Z_k(K,Q,\nu)$). In addition, suppose we have a function $F:\RR\to \RR$ and a function $v\in C(K)$. We write, for $\mu \in \mathcal M(K)$, 
$$<v,\mu
>:= \int_K v d\mu
$$
and then 
\begin{equation}\label{lambdaint}\int_{\mathcal M(K)} F(<v,\mu
>)d\sigma_k(\mu
):=\end{equation}
$$\frac{1}{Z_k} \int_K \cdots \int_K  |VDM_k^Q(x_1,...,x_{N_k
})|^2  F\bigl(\frac{1}{N_k
}\sum_{j=1}^{N_k
} v(x_j)\bigr)d\nu(x_1) \cdots d\nu(x_{N_k
}).$$

\begin{theorem} \label{ldp} The sequence $\{\sigma_k=(j_k)_*(Prob_k)\}$ of probability measures on $\mathcal M(K)$ satisfies a {\bf large deviation principle} with speed $2kN_k$ and good rate function $\mathcal I:=\mathcal I_{K,Q}$ where, for $\mu \in \mathcal M(K)$,
\begin{equation}\label{ratefcnlform}\mathcal I(\mu):=\log J^Q(\mu_{K,Q})-\log J^Q(\mu)=\log W^Q(\mu_{K,Q})-\log W^Q(\mu).\end{equation}
 \end{theorem}

This means that $\mathcal I:\mathcal M(K)\to [0,\infty]$ is a lowersemicontinuous mapping such that the sublevel sets $\{\mu \in \mathcal M(K): \mathcal I(\mu)\leq \alpha\}$ are compact in the weak-* topology on $\mathcal M(K)$ for all $\alpha \geq 0$ ($\mathcal I$ is ``good'') satisfying (\ref{lowb}) and (\ref{highb}):

\begin{definition} \label{equivform}
The sequence $\{\mu_k\}$ of probability measures on $\mathcal M(K)$ satisfies a {\bf large deviation principle} (LDP) with good rate function $\mathcal I$ and speed $2kN_k$ if for all 
measurable sets $\Gamma\subset \mathcal M(K)$, 
\begin{equation}\label{lowb}-\inf_{\mu \in \Gamma^0}\mathcal I(\mu)\leq \liminf_{k\to \infty} \frac{1}{2kN_k} \log \mu_k(\Gamma) \ \hbox{and}\end{equation}
\begin{equation}\label{highb} \limsup_{k\to \infty} \frac{1}{2kN_k} \log \mu_k(\Gamma)\leq -\inf_{\mu \in \bar \Gamma}\mathcal I(\mu).\end{equation}
\end{definition}

In the setting of $\mathcal M(K)$, to prove a LDP it suffices to work with a base for the weak-* topology. The following is a special case of a basic general existence result for a LDP given in Theorem 4.1.11 in \cite{DZ}.

\begin{proposition} \label{dzprop1} Let $\{\sigma_{\epsilon}\}$ be a family of probability measures on $\mathcal M(K)$. Let $\mathcal B$ be a base for the topology of $\mathcal M(K)$. For $\mu\in \mathcal M(K)$ let
$$\mathcal I(\mu):=-\inf_{\{G \in \mathcal B: \mu \in G\}}\bigl(\liminf_{\epsilon \to 0} \epsilon \log \sigma_{\epsilon}(G)\bigr).$$
Suppose for all $\mu\in \mathcal M(K)$,
$$\mathcal I(\mu):=-\inf_{\{G \in \mathcal B: \mu \in G\}}\bigl(\limsup_{\epsilon \to 0} \epsilon \log \sigma_{\epsilon}(G)\bigr).$$
Then $\{\sigma_{\epsilon}\}$ satisfies a LDP with rate function $\mathcal I(\mu)$ and speed $1/\epsilon$. 
\end{proposition}

We give our first proof of Theorem \ref{ldp} using Theorem \ref{obsolete}. 

\begin{proof} As a base $\mathcal B$ for the topology of $\mathcal M(K)$, we can take the sets from (\ref{nbhdbase}) or simply all open sets. For $\{\sigma_{\epsilon}\}$, we take the sequence of probability measures $\{\sigma_k\}$ on $\mathcal M(K)$ and we take $\epsilon =\frac{1}{2kN_k}$. For $G\in \mathcal B$, 
$$\frac{1}{2kN_k}\log \sigma_k(G)= \log J_k^Q(G)-\frac{1}{2kN_k}\log Z_k$$
using (\ref{jkqmu}) and (\ref{sigmak}). From Proposition \ref{weightedtd}, and (\ref{jeqn}) with $v=Q$, 
$$\lim_{k\to \infty} \frac{1}{2kN_k}\log Z_k=\log \bar \delta^Q(K)= \log J^Q(\mu_{K,Q});$$ and by Theorem \ref{obsolete}, if $\inf_{v} [\log \bar \delta^{v} (K) +\int_K vd\mu]>-\infty$,
$$\inf_{G \ni \mu} \limsup_{k\to \infty} \log J_k^Q(G)=\inf_{G \ni \mu} \liminf_{k\to \infty} \log J_k^Q(G)=\log J^Q(\mu).$$
If, on the other hand, $\inf_{v} [\log \bar \delta^{v} (K) +\int_K vd\mu]=-\infty$, then 
$$\lim_{k\to \infty} \log J_k^Q(G)=-\infty.$$
Thus by Proposition \ref{dzprop1} and Theorem \ref{obsolete}, $\{\sigma_k\}$ satisfies an LDP with rate function 
$$\mathcal I(\mu):=\log J^Q(\mu_{K,Q})-\log J^Q(\mu)=\log W^Q(\mu_{K,Q})-\log W^Q(\mu)$$
and speed $2kN_k$. This rate function is good since $\mathcal M(K)$ is compact.
\end{proof}

Proposition \ref{cor44} shows that $\mu_{K,Q}$ maximizes the functional 
\begin{equation} \label{wtden} \mu \to \log W^Q(\mu) \end{equation}
over all $\mu\in {\mathcal M}(K)$. Thus
\begin{equation} \label{ldpmin}\mathcal I_{K,Q}(\mu)\geq 0 \ \hbox{with} \ \mathcal I_{K,Q}(\mu)= 0 \iff \mu=\mu_{K,Q}. \end{equation}
To summarize, $\mathcal I_{K,Q}$ is a good rate function with unique minimizer $\mu_{K,Q}$.

There is a converse to Proposition \ref{dzprop1}, Theorem 4.1.18 in \cite{DZ}. For $\mathcal M(K)$, it reads as follows:

\begin{proposition} \label{dzprop2} Let $\{\sigma_{\epsilon}\}$ be a family of probability measures on $\mathcal M(K)$. Suppose that $\{\sigma_{\epsilon}\}$ satisfies a LDP with rate function $\mathcal I(\mu)$ and speed $1/\epsilon$. Then for any base $\mathcal B$ for the topology of $\mathcal M(K)$  and any $\mu\in \mathcal M(K)$ 
$$\mathcal I(\mu):=-\inf_{\{G \in \mathcal B: \mu \in G\}}\bigl(\liminf_{\epsilon \to 0} \epsilon \log \sigma_{\epsilon}(G)\bigr)$$
$$=-\inf_{\{G \in \mathcal B: \mu \in G\}}\bigl(\limsup_{\epsilon \to 0} \epsilon \log \sigma_{\epsilon}(G)\bigr).$$ 
\end{proposition}

\begin{remark} \label{equival} This shows that, starting with a strong Bernstein-Markov measure $\nu$ and the corresponding sequence of probability measures $\{\sigma_k\}$ on $\mathcal M(K)$ in (\ref{sigmak}), the existence of 
an LDP with rate function $\mathcal I(\mu)$ and speed $2kN_k$ {\it implies} that
necessarily
\begin{equation} \label{ourrate} \mathcal I(\mu)=\log J^Q(\mu_{K,Q})-\log J^Q(\mu).\end{equation} 
We mention that uniqueness of the rate function is basic (cf., Lemma 4.1.4 of \cite{DZ}). The main ingredient utilized here is Theorem \ref{frombbn} on the distribution of asymptotic weighted Fekete points. This in turn hinges on the differentiability result Theorem \ref{bbdiff}. On the other hand, in potential-theoretic situations where the analogue of Theorem \ref{frombbn} is easily obtained, this approach to a LDP is fairly straightforward. We illustrate this in the univariate case in Section \ref{sec:fr}. 

\end{remark}

We turn to the second proof of Theorem \ref{ldp}. This follows from Corollary 4.6.14 in \cite{DZ}, which is a general version of the G\"artner-Ellis theorem. This approach was brought to our attention by S. Boucksom. Again, we state the version of the \cite{DZ} result for $\mathcal M(K)$.

\begin{proposition} \label{gartell}
Let $\{\sigma_{\epsilon}\}$ be a family of Borel probability measures on $C(K)^*$ equipped with the weak-* topology. Suppose for each $\lambda \in C(K)$, the limit
$$\Lambda(\lambda):=\lim_{\epsilon \to 0} \epsilon \log \int_{C(K)^*} e^{\lambda(x)/\epsilon}d\sigma_{\epsilon}(x)$$
exists as a finite real number and assume $\Lambda$ is Gateaux differentiable; i.e., for each $\lambda, \theta\in C(K)$, the function $f(t):= \Lambda(\lambda +t\theta)$ is differentiable at $t=0$. Then $\{\sigma_{\epsilon}\}$ satisfies an LDP in $C(K)^*$ with the convex, good rate function $\Lambda^*$.
\end{proposition}

Here $$\Lambda^*(x):= \sup_{\lambda \in C(K)}\bigl(<\lambda,x>- \Lambda(\lambda)\bigr),$$ is the Legendre transform of $\Lambda$. The upper bound (\ref{highb}) in the LDP holds with rate function $\Lambda^*$ under the assumption that the limit $\Lambda(\lambda)$ exists and is finite; the Gateaux differentiability of $\Lambda$ is needed for the lower bound (\ref{lowb}). 
We proceed with the second proof of Theorem \ref{ldp}.

\begin{proof} We show that for each $v\in C(K)$,
$$\lim_{k \to \infty} \frac{1}{2kN_k
}\log  \int_{C(K)^*} e^{2kN_k
<v,\mu>}d\sigma_{k}(\mu)$$ exists as a finite real number. First, since $\sigma_k$ is in ${\mathcal M}(K)$, the integral can be taken over ${\mathcal M}(K)$. Consider
$$\frac{1}{2kN_k
}\log  \int_{{\mathcal M}(K)} e^{2kN_k
<v,\mu>}d\sigma_{k}(\mu).$$
By (\ref{lambdaint}), this is equal to
$$\frac{1}{2kN_k
}\log  \frac{1}{Z_k}\cdot \int_{K^{N_k
}}  |VDM_k^{Q-v}(x_1,...,x_{N_k
})|^2 d\nu(x_1) \cdots d\nu(x_{N_k
}).$$
Recall that
$$Z_k= \int_{K^{N_k
}} |VDM_k^Q(x_1,...,x_{N_k
}))|^2d\nu(x_1) \cdots d\nu(x_{N_k
}).$$
Define
$$\tilde Z_k:= \int_{K^{N_k
}}  |VDM_k^{Q-v}(x_1,...,x_{N_k
})|^2 d\nu(x_1) \cdots d\nu(x_{N_k
}).$$
Then we have
$$\lim_{k\to \infty} \tilde Z_k^{\frac{1}{2kN_k
}}= \bar \delta^{Q-v} (K) \ \hbox{and} \ \lim_{k\to \infty}  Z_k^{\frac{1}{2kN_k
}}= \bar \delta^Q (K) $$
from (\ref{zeen}) in Proposition \ref{weightedtd} and the assumption that $(K,\nu,\tilde Q)$ satisfies the weighted Bernstein-Markov property for {\it all} $\tilde Q\in  C(K)$. Thus
\begin{equation}\label{lambda}\Lambda(v)= \lim_{k \to \infty} \frac{1}{2kN_k
}\log \frac{\tilde Z_k}{Z_k}=  \log \frac{\bar \delta^{Q-v} (K)}{ \bar \delta^{Q} (K)} .\end{equation}
Define now, for $v, v'\in C(K)$, 
$$f(t):=  E(V_{K,Q-(v+tv')}).$$
Theorem \ref{bbdiff} and (\ref{keyrel}) give that $\Lambda$ is Gateaux
 differentiable and Proposition \ref{gartell} gives that $\Lambda^*$ is a rate function on $C(K)^*$.
 
 Since each $\sigma_k$ has support in $\mathcal M(K)$, it follows from (\ref{lowb}) and (\ref{highb}) in Definition \ref{equivform} of an LDP with $\Gamma \subset C(K)^*$ that for $\mu\in C(K)^*\setminus \mathcal M(K)$, $\Lambda^*(\mu)=+\infty$. By Lemma 4.1.5 (b) of \cite{DZ}, the restriction of $\Lambda^*$ to $\mathcal M(K)$ is a rate function. Since $\mathcal M(K)$ is compact, it is a good rate function. Being a Legendre transform, $\Lambda^*$ is convex.  

To compute $\Lambda^*$ and demonstrate (\ref{ratefcnlform}), we have, using (\ref{lambda}) and (\ref{keyrel}),
$$\Lambda^*(\mu)= \sup_{v\in C(K)} \bigl( \int_K v d\mu - E(V_{K,Q})+E(V_{K,Q-v})\bigr).$$
Thus $$\Lambda^*(\mu)+E(V_{K,Q})= \sup_{v\in C(K)} \bigl( \int_K v d\mu +E(V_{K,Q-v})\bigr)$$
 $$= \sup_{u\in C(K)} \bigl( E(V_{K,Q+u})-\int_K u d\mu \bigr) \ (\hbox{taking} \ u=-v).$$
Rearranging and replacing $u$ in the supremum by $v=u+Q$,
$$\Lambda^*(\mu)=  \sup_{u\in C(K)} \bigl( E(V_{K,Q+u})-\int_K u d\mu \bigr)-E(V_{K,Q})$$
\begin{equation} \label{ratefcnl} = \sup_{v\in C(K)} \bigl( E(V_{K,v})-\int_K v d\mu \bigr)-\bigr[E(V_{K,Q})-\int_K Q d\mu\bigr].\end{equation}
Using (\ref{keyrel}) and Theorem \ref{obsolete}, we obtain (\ref{ratefcnlform}).

\end{proof}

\begin{remark} \label{equival2} From the definition of $E^*$ in Definition \ref{eleven}, equation (\ref{ratefcnl}) gives, after some manipulation, 
\begin{equation}\label{bermanrate}\Lambda^*(\mu)=E^*(\mu)+\int_K (Q-u_0)d\mu -E(V_{K,Q}).\end{equation}
Recall from (\ref{wtenmin}) that $\mu_{K,Q}$ minimizes the functional
$$\mu \to E^*(\mu)+\int_K (Q-u_0)d\mu$$ and the minimal value is $E(V_{K,Q})$. Thus $\Lambda^*(\mu)\geq 0$ and $\Lambda^*(\mu)=0$ precisely for $\mu=\mu_{K,Q}$. In this way one obtains the LDP with rate function arising from the functional $E^*$ in a very natural manner as this functional is itself a Legendre-type transform (see \cite{B}). An advantage of using Proposition \ref{gartell} is that one automatically gets {\it convexity} of the rate functional. Here, the asymptotic weighted Fekete result Theorem \ref{frombbn} is not needed, but one clearly uses the differentiability result Theorem \ref{bbdiff}. Finally, to deduce the equalities
$$-\log \overline J(\mu)= -\log \underline J(\mu)=E^*(\mu)-\int_Ku_0d\mu-E(V_T)$$ and 
   $$-\log \overline J^Q(\mu)= -\log \underline J^Q(\mu)=E^*(\mu)+\int_K (Q-u_0)d\mu-E(V_T),$$
one can appeal to the Rumely-type formula (\ref{keyrel}). Alternately, starting with a strong Bernstein-Markov measure $\nu$ and the corresponding sequence of probability measures $\{\sigma_k\}$ on $\mathcal M(K)$ in (\ref{sigmak}), the G\"artner-Ellis approach gives the LDP with rate function $\Lambda^*$ in (\ref{bermanrate}) using Proposition \ref{weightedtd} and the differentiability result Theorem \ref{bbdiff}. Then the {\it uniqueness} of the rate function gives, as mentioned in Remark \ref{equival}, that this rate function $\Lambda^*$ is given by $\mathcal I$ in (\ref{ourrate}). Thus one obtains the equalities above relating $E^*,J$ and $J^Q$.  
\end{remark}

The differentiability result Theorem \ref{bbdiff} is crucial in either proof of the LDP. In the Angelesco ensemble case studied in \cite{blang} as well as more general univariate vector energy settings, a natural formulation of an analogue of this result is unclear. However, in these settings, the analogue of the asymptotic weighted Fekete result Theorem \ref{frombbn} is straightforward. We illustrate the ease in obtaining an LDP in the next section in the classical weighted potential-theoretic setting in one variable.

\section{\bf The univariate situation.}\label{sec:fr}
Let $K\subset \CC$ be compact and non-polar. For $\mu\in \mathcal M (K)$, we write 
$$p_{\mu}(z):=\int_K \log \frac{1}{|z-\zeta|}d\mu(\zeta),$$
the logarithmic potential function of $\mu$; 
$$I(\mu):=\int_K \int_K \log \frac{1}{|z-\zeta|}d\mu(\zeta)d\mu(z),$$
the logarithmic energy of $\mu$; and, if $Q\in \mathcal A(K)$, 
$$I^Q(\mu):= I(\mu)+2\int_K Qd\mu$$
is the weighted logarithmic energy of $\mu$ with respect to the weight $Q$. The measure $\mu_{K,Q}=dd^cV^*_{K,Q}$ satisfies $\inf_{\nu \in \mathcal M(K)} I^Q(\nu)=I^Q(\mu_{K,Q})$. For compact sets and admissible weights in $\CC$, (\ref{wtdtd}) and Theorem \ref{frombbn} are standard facts, requiring no more effort than the classical unweighted case (cf., \cite{ST}). The version of Corollary \ref{largedev} proved in \cite{BLtd} follows immediately from the univariate version of Proposition \ref{weightedtd}. Theorem \ref{frombbn} immediately yields Proposition \ref{cor44} in this setting. In addition, the version of the domination principle used in the proof of Theorem \ref{wobsolete} in the univariate situation is Theorem 3.2 in Chapter II of \cite{ST}. 

Thus, a priori, the only nonstandard ingredient required to prove Theorems \ref{wobsolete}, \ref{obsolete} and \ref{ldp} in $\CC$ is (\ref{needit}). We show this result also follows from standard potential-theoretic arguments. We begin with $\mu\in \mathcal M (K)$ having $I(\mu)< +\infty$ and we let $u(z):=-p_{\mu}(z)$. As in the proof of Theorem \ref{wobsolete}, the problem is that $u$ is only usc and thus $\mu$ is not necessarily of the form $\mu_{K,Q}$ for $Q\in \mathcal A(K)$. Proceeding as in Theorem \ref{wobsolete}, we take a sequence of continuous admissible weights $\{Q_j\}$ with $Q_j \downarrow u$ on $K$. By the domination principle, we get   \begin{equation}\label{dummy1}u_j:=V^*_{K,Q_j} \downarrow u =-p_{\mu} \  \hbox{on}  \ \CC. \end{equation} We have $\mu_j :=\mu_{K,Q_j}\to \mu$ weak-*. Note that $dd^cu_j=-dd^c p_{\mu_j}=\mu_j$ implies
  \begin{equation}\label{dummy}p_{\mu_j}+u_j=F_j \ \hbox{on} \ \CC \end{equation}
 for some constant $F_j$. Finally, for $z\not \in K$, the function $\zeta \to \log |z-\zeta|$ is continuous on $K$ and thus
  $\mu_j \to \mu$ weak-* implies that
  \begin{equation}\label{last4}\lim_{j\to \infty} p_{\mu_j}(z)=p_{\mu}(z) \ \hbox{for} \ z\in \CC\setminus K.\end{equation}
 
 \begin{proposition}\label{laststep} With the preceding notation, 
 $$\lim_{j\to \infty} F_j =0; \ \lim_{j\to \infty} I(\mu_j) = I(\mu); \ \hbox{and} \
\lim_{j\to \infty} \int_K u_j d\mu_j = \int_K ud\mu. $$
\end{proposition} 

\begin{proof} That $\lim_{j\to \infty} F_j =0$ follows immediately from (\ref{dummy1}), (\ref{dummy}) and (\ref{last4}). Noting that $I(\mu_j)=\int_K p_{\mu_j}d\mu_j$, this fact and (\ref{dummy}) imply that the conditions
$$\lim_{j\to \infty} I(\mu_j) = I(\mu) \ \hbox{and} \
\lim_{j\to \infty} \int_K u_j d\mu_j = \int_K ud\mu $$
are equivalent.

Since $u_{j+1}\leq u_j$, for all $\nu\in {\mathcal M}(K)$
$$I(\nu)+2\int_K u_{j+1} d\nu\leq I(\nu)+2\int_K u_{j} d\nu; \ \hbox{thus} \ I^{Q_{j+1}}(\mu_{j+1}) \leq I^{Q_j}(\mu_j)$$
and $\lim_{j\to \infty}I^{Q_j}(\mu_j)=\lim_{j\to \infty}[I(\mu_j) +2\int_K u_j d\mu_j ]$ exists. 

We claim that for all $\nu\in {\mathcal M}(K)$
 \begin{equation}\label{dummy11}I(\nu)+2\int_K u d\nu\geq I(\mu)+2\int_K u d\mu.\end{equation}
Clearly we need only show this for $\nu$ with $I(\nu)<+\infty$. Since $I(\mu)<+\infty$, and $\mu-\nu$ is a signed measure of total mass $0$ and of compact support, $I(\mu-\nu)\geq 0$ (cf., Lemma I.1.8 in \cite{ST}). Using 
$$-\int_K ud\nu=\int_K p_{\mu}d\nu=\int_K p_{\nu}d\mu \ \hbox{and} \ I(\mu)=\int_K p_{\mu}d\mu = -\int_K ud\mu$$ we have
$$0\leq I(\mu-\nu)= I(\mu) +I(\nu)-2\int_K p_{\mu}d\nu= I(\mu) +I(\nu)+2\int_K ud\nu$$
so that 
$$I(\nu)+2\int_K u d\nu\geq -I(\mu)=I(\mu) +2\int_K ud\mu$$
as desired. 

Since $u_j\geq u$,
 \begin{equation}\label{dummy8}I(\mu_j)+2\int_K u_j d\mu_j\geq I(\mu_j)+2\int_K u d\mu_j\geq I(\mu)+2\int_K u d\mu\end{equation}
 where the last inequality comes from taking $\nu = \mu_j$ in (\ref{dummy11}).

We have $I^{Q_j}(\mu_j)\leq I^{Q_j}(\mu)$ (since $\mu_j$ minimizes this weighted energy); by monotone convergence using (\ref{dummy1}),
$$I(\mu) +2\int_K u d\mu = \lim_{j\to \infty} [I(\mu)+ 2\int_K u_j d\mu]\geq  \lim_{j\to \infty} [I(\mu_j)+ 2\int_K u_j d\mu_j].$$
 Combining this with (\ref{dummy8}), we have
\begin{equation}\label{dummy12}I(\mu) +2\int_K u d\mu = \lim_{j\to \infty} [I(\mu)+ 2\int_K u_j d\mu]=\lim_{j\to \infty} [I(\mu_j)+ 2\int_K u_j d\mu_j].\end{equation}
We already know that
\begin{equation}\label{dummy13}\lim_{j\to \infty} F_j=\lim_{j\to \infty} [I(\mu_j)+ \int_K u_j d\mu_j]=0;\end{equation}
subtracting (\ref{dummy13}) from (\ref{dummy12}) gives the result:
$$\lim_{j\to \infty}\int_K u_j d\mu_j=I(\mu) +2\int_K u d\mu =\int_K u d\mu.$$
\end{proof}

As in \cite{bloomvoic}, our functionals $J(\mu),W(\mu)$ are related to $I(\mu)$.
\begin{proposition}\label{wisi} Let $\mu \in \mathcal M(K)$. Then 
\begin{equation}\label{iisw}-\log W(\mu)=-\log J(\mu)=\frac{1}{2}I(\mu).
\end{equation}
\end{proposition}
\begin{proof} Observe that if $I(\mu)<+\infty$, the logarithmic potential function $p_{\mu}$ satisfies $u:=-p_\mu\in L(\CC)$ with $dd^cu =\mu$ and $\int_K u d\mu >-\infty$. In particular, if $\mu=\mu_{K,Q}$ for $Q\in  C(K)$, (\ref{newstep}) and Theorem III.1.3 of \cite{ST} show that 
$$-\log W(\mu_{K,Q})=\frac{1}{2}I(\mu_{K,Q})<+\infty.$$
For $\mu \in \mathcal M(K)$ with $I(\mu)<+\infty$, as in the proof of Theorem \ref{wobsolete}, there exists a sequence of continuous admissible weights $\{Q_j\}$ with $Q_j \downarrow -p_\mu$ on $K$, $\mu_{K,Q_j}\to \mu$ weak-* and 
 $$\lim_{j\to \infty} W(\mu_{K,Q_j})=W(\mu)$$
Proposition \ref{laststep} gives (\ref{iisw}) in this case. It remains to show that 
 \begin{equation}\label{last}I(\mu)=+\infty \ \hbox{implies} \ W(\mu)=0.\end{equation} The proof of Theorem 4.1 (a) of \cite{bloomvoic} that $\log W(\mu)\leq -\frac{1}{2}I(\mu)$ follows directly from the definition of these functionals and usc of $(\zeta,z)\to \log |z-\zeta|$. Thus (\ref{last}) holds.

\end{proof}

\begin{corollary} Let $Q\in  C(K)$. For $\mu \in \mathcal M(K)$
\begin{equation}\label{iqisqw}-\log W^Q(\mu)=-\log J^Q(\mu)=\frac{1}{2}I(\mu)+\int_KQd\mu=\frac{1}{2}I^Q(\mu).
\end{equation}

\end{corollary}

Note that the rate function is 
$$\mathcal I_{K,Q}(\mu)=\frac{1}{2}[\log I^Q(\mu)-\log I^Q(\mu_{K,Q})]$$
and the speed is $2k^2$.

\begin{remark} \label{gueremark} The GUE has joint probability distribution 
$Prob_k$ on $\RR^{k+1}$ given by
$$Prob_k(A):=\frac{1}{Z_k}\cdot \int_A  |VDM_k^Q(x_1,...,x_{k+1
})|^2 \cdot dx_1\cdots dx_{k+1}$$
for a Borel set $A\subset \RR^{k+1}$ where $Q(x)=x^2$ and
$$Z_k= \int_{\RR^{k+1}}  |VDM_k^Q(x_1,...,x_{k+1
})|^2 \cdot dx_1\cdots dx_{k+1}.$$
Here, $\mu_{\RR,Q}=\frac{1}{\pi}\sqrt{2-x^2}dx$ on the interval $-\sqrt 2 \leq  x\leq \sqrt 2$.
\end{remark}

\end{document}